\documentclass[12pt]{amsart}

\newtheorem*{citethm}{Theorem}
\newtheorem{thm}{Theorem}

\newtheorem{prop}{Proposition}

\theoremstyle{definition}
\newtheorem*{rem}{Remark}
\newtheorem*{ex}{Example}

\providecommand{\RR}{\mathbb{R}}
\providecommand{\CC}{\mathbb{C}}

\providecommand{\ZZ}{\mathbb{Z}}
\providecommand{\NN}{\mathbb{N}}
\DeclareMathOperator{\sym}{Sym}
\providecommand{\g}{{\mathfrak{g}}}

\begin{document}
\title{Maximal quadratic modules on $\ast$-rings}{}
\author{J. Cimpri\v c}
\date{April 6th 2005}
\address{Cimpri\v c Jakob, University of Ljubljana, Faculty of Mathematics and Physics,
Department of Mathematics, Jadranska 19, SI-1000 Ljubljana, Slovenia,
cimpric@fmf.uni-lj.si,http://www.fmf.uni-lj.si/\~{}cimpric}

\maketitle

\begin{abstract}
We generalize the notion of and results on maximal proper quadratic modules 
from commutative unital rings to $\ast$-rings and discuss the relation of this 
generalization to recent developments in noncommutative real algebraic geometry.
The simplest example of a maximal proper quadratic module is the cone 
of all positive semidefinite complex matrices of a fixed dimension.
We show that the support of a maximal proper quadratic module
is the symmetric part of a prime $\ast$-ideal, that every maximal
proper quadratic module in a Noetherian $\ast$-ring comes from
a maximal proper quadratic module in a simple artinian ring with involution
and that maximal proper quadratic modules satisfy an intersection theorem.
As an application we obtain the following extension of Schm\" udgen's 
Strict Positivstellensatz for the Weyl algebra: 
Let $c$ be an element of the Weyl algebra $\mathcal{W}(d)$ 
which is not negative semidefinite in the Schr\" odinger representation.
It is shown that under some conditions there exists an integer $k$
and elements $r_1,\ldots,r_k \in \mathcal{W}(d)$ such that
$\sum_{j=1}^k r_j c r_j^\ast$ is a finite sum of hermitian squares.
This result is not a proper generalization however because we
don't have the bound $k \le d$. 
\end{abstract}

\section{Introduction}
\label{firstsec}

The aim of this note is to generalize the notion of and results on quadratic modules from 
commutative unital rings to associative unital rings with involution, which we call $\ast$-rings.
The study of quadratic modules in $\ast$-rings is suggested by the recent developments in 
noncommutative real algebraic geometry, see \cite{av}, \cite{h1}, \cite{pc}, \cite{sch1}, \cite{sch2}.

Commutative real algebraic geometry is based on the notion of an
ordering and quadratic modules are considered just  a technical tool.
However, an attempt by M. Marshall to build a noncommutative 
real algebraic geometry on $\ast$-orderings in \cite{mm} showed
that there is not enough of them. The advantage of maximal proper
quadratic modules over $\ast$-orderings is not only in their quantity
but also in their connection to the representation theory
of $\ast$-rings, see \cite{cstar}.

The exposition can be divided into three parts.
In the first part (Sections \ref{defsec} and \ref{suppsec}) we define
a quadratic module in a $\ast$-ring and provide elementary examples.
We also try to generalize the following result
(cf. \cite[1.1.4 Theorem]{walter}): If $M$ is a maximal proper quadratic
module in a commutative ring $R$ with trivial involution, then
$M \cup -M = R$ and $M \cap -M$ is a prime ideal. It will
be shown that the first property cannot be generalized but the second can.

In the second part (Sections \ref{fracsec} and \ref{repsec}) we show that
every maximal quadratic module in a Noetherian $\ast$-ring comes from
a maximal quadratic module in a simple artinian ring with involution
using a variant of Goldie's theory from \cite{domo}. In the commutative
case, this is rather obvious. Namely, by factoring out the prime ideal $M \cap -M$
we get a maximal proper quadratic module in $R/M \cap -M$ and by passing
to the field $F$ of fraction of $R/M \cap -M$ we get a maximal proper quadratic
module in $F$, both times in a natural way.

In the third part (Sections \ref{intsec} and \ref{appsec}) we generalize the following
intersection theorem (see \cite[1.8 Satz]{jacobi}): An element $r$ of a commutative ring
$R$ with trivial involution belongs to $R \setminus -M$ for every maximal proper quadratic
module $M$ in $R$ if and only if there exist elements $m,t \in R$ which are sums of squares 
of elements from $R$ and satisfy $tr=1+m$. As an application, we obtain the following
extension of Schm\" udgen's Strict Positivstellensatz for the Weyl algebra
(see \cite[Theorem 1.1]{sch1}): 

\begin{citethm} 
Let $\mathcal{W}(d)$ be the $d$-th Weyl algebra with the natural involution
and let $\pi_0$ be its Schr\" odinger representation. If $c$ is a symmetric element
of $\mathcal{W}(d)$ of even degree $2m$ then the following are equivalent:
\begin{enumerate}
\item $\pi_0(c)$ is not negative semidefinite and 
the highest degree part $c_{2m}(z,\overline{z})$ of $c$
is  strictly positive  for all $z \in \mathbb{C}^d$, $z \ne 0$.
\item There exist elements $r_1,\ldots,r_k, s_0,s_1, \ldots,s_l \in \mathcal{W}(d)$
such that 
$
\sum_{j=1}^k r_j c r_j^\ast = \sum_{i=0}^l s_i s_i^\ast
$
and $\pi_0(s_0)$ is invertible and $\deg(s_0) \ge \deg(s_j)$ for every $j=1,\ldots,l$.
\end{enumerate}
\end{citethm}
An early version of this manuscript was presented in Luminy and Saskatoon in March 2005.
Meanwhile, our techniques have also been applied to the free $\ast$-algebra, see \cite{ks}.

\section{Definitions and elementary examples}
\label{defsec}

Let $R$ be a $\ast$-ring 
and $\sym(R) := \{r \in R \vert \ r=r^\ast\}$ its set of symmetric elements. 
When no confusion is possible we write $S$ for $\sym(R)$.
A subset $M$ of $R$ is a \textit{quadratic module} if $M \subseteq S$, 
$1 \in M$, $M+M \subseteq M$ and $rMr^\ast \subseteq M$ for every $r \in R$.
(This is very similar to the definition of an \textit{m-admissible wedge} 
in \cite[page 22]{schbook}.)
A quadratic module $M$ is \textit{proper} if $-1 \not\in M$. The smallest quadratic
module in $R$ is $N(R) := \{\sum_i r_i r_i^\ast \ \colon \ r_i \in R\}$. 
The quadratic module $N(R)$ need not be proper. If $N(R)$ is proper,
then $R$ is \textit{semireal}.
A proper quadratic module is \textit{maximal} if it is not contained in any
strictly larger proper quadratic module.

\begin{rem} The following properties of $R$ are equivalent:
\begin{enumerate}
\item $R$ is semireal, 
\item $R$ has at least one proper quadratic module,
\item $R$ has at least one maximal proper quadratic module.
\end{enumerate}
\end{rem}

\begin{ex} Let $H$ be a complex Hilbert space and $B(H)$ the algebra of all bounded operators on $H$
with the standard involution. Then the set $P(H)$ of all positive definite operators from $H$ is clearly a 
proper quadratic module in $B(H)$. 

If $H$ is finite dimensional, then $P(H)$ is maximal.
Namely, if $A$ is a symmetric operator which is not in $P(H)$ then we can find a basis
of $H$ such that the matrix of $A$ is diagonal with first entry equal to $-1$. Let $E_{ij}$ be the
operator whose matrix in this basis has $(i,j)$-th entry equal to $1$ and all other entries equal to zero.
Note that $\sum_k E_{k1} A E_{k1}^\ast = -I$ where $I$ is the identity operator on $H$.
Hence, every quadratic module that contains $A$ also contains $-I$. 
Therefore, there is no proper quadratic module stricly larger then $P(H)$. 
Note that in the finite dimensional case $P(H) = N(B(H))$,
hence $P(H)$ is the only proper quadratic module in $B(H)$.

If $H$ is infinite dimensional, then $P(H)$ is not maximal.
Namely if $K_s$ is the set of all symmetric compact operators on $H$, then $M = P(H)+K_s$
is a proper quadratic module in $B(H)$ that is strictly larger than $P(H)$. If $-1 \in M$ then
there exists $A \in P(H)$ such that $-I-A$ is compact. This is not possible because the eigenvalues
of a compact operator tend to zero while the eigenvalues of $-I-A$ are bounded away from zero.
\end{ex}

\begin{rem}
If $R$ is a commutative unital ring with trivial involution (i.e. $\sym(R)=R$) 
and $M$ is a maximal proper quadratic module in $R$, then $M \cup -M = \sym(R)$. 
This property fails in the noncommutative case. Namely, if $H$ is a finite dimensional 
Hilbert space of dimension at least two then $P(H)$ is a maximal proper quadratic module 
on $B(H)$ such that $P(H) \cup -P(H) \ne \sym(B(H))$. 
\end{rem}

\section{The support}
\label{suppsec}

Let $M$ be a proper quadratic module in a $\ast$-ring $R$. The set
\[
M^0 = M \cap -M
\]
is called \textit{the support} of $M$. We will frequently use the following properties of $M^0$:
$M^0+M^0 \subseteq M^0$, $-M^0 \subseteq M^0$, $r M^0 r^\ast \subseteq M^0$ for every
$r \in R$, if $x+y \in M^0$ and $x,y \in M$ then $x \in M^0$ and $y \in M^0$. Write 
\[
J_M = \{ a \in R \vert \ a u u^\ast a^\ast \in M^0 \text{ for all } u \in R\}.
\]

\begin{prop}
If $M$ is a proper quadratic module in $R$, then $J_M$ is a two-sided
ideal in $R$ containing the set $M^0$.
\end{prop}

\begin{proof}
If $a,b \in J_M$ then for every $u \in R$, $a u u^\ast a^\ast \in M^0$ and 
$b u u^\ast b^\ast \in M^0$. Write $x = (a+b) u u^\ast (a+b)^\ast$
and $y = (a-b) u u ^\ast (a-b)^\ast$. Since $x+y =2(a u u^\ast a^\ast 
+b u u^\ast b^\ast) \in M^0$ and $x,y \in N(R) \subseteq M$, it follows
that $x \in M^0$ and $y \in M^0$ for any $u \in R$. Hence $a+b \in J_M$
and $a-b \in J_M$.

If $a \in J_M$ and $r \in R$ then for every $u \in R$,
$(ra) u u^\ast (r a)^\ast = r (a u u^\ast a^\ast) r^\ast \in M^0$ and
$(ar) u u^\ast (a r)^\ast = a (ru)(ru)^\ast a^\ast \in M^0$, so that
$ra \in J_M$ and $ar \in J_M$. Hence, $J_M$ is a two-sided ideal.

To prove that $M^0 \subseteq J_M$, we must show that
$a u u^\ast a \in M^0$ for every $a \in M^0$ and $u \in R$.
Pick $a \in M^0$ and $u \in R$ and write $z=uu^\ast$,
$x = (1+az)a(1+az)^\ast$ and $y = (1-az)a(1-az)^\ast$.
Note that $x,y \in M^0$ and $4aza = x-y \in M^0$.
Since $aza \in M$, it follows that $aza \in M^0$.
\end{proof}

Recall that a two-sided ideal $J$ of a ring $R$ is \textit{prime}
if for any $a,b \in R$ such that $aRb \subseteq J$ we have that
either $a \in R$ or $b \in R$. A ring $R$ is \textit{prime} if
$(0)$ is a prime ideal of $R$.

\begin{thm}
\label{prime}
If $M$ is a maximal proper quadratic module in $R$,
then the ideal $J_M$ is prime and $\ast$-invariant.
Moreover, $J_M \cap S = M^0$.
\end{thm}

\begin{proof}
We already know that $J_M$ is a two-sided ideal containing $M^0$.
Hence, $M^0 \subseteq J_M \cap S$. If $J_M \cap S \not\subseteq M^0$, 
then there exists $s \in J_M \cap S$ such that $s \not\in M^0$.
Replacing $s$ by $-s$ if necessary, we may assume that $-s \not\in M$.
Since $M$ is maximal, it follows that the smallest quadratic module
containing $M$ and $-s$ is not proper. Hence there exist an element
$n \in M$, an integer $l$ and elements $t_1,\ldots,t_l \in R$ such that
$1+n = \sum_{j=1}^l t_j s t_j^\ast$. Since $J_M$ is a two-sided ideal
containing $s$, it follows that $1+n \in J_M$. It follows that
$(1+n) v v^\ast (1+n) \in M^0$ for every $v \in R$. For $v=1$,
we get $(1+n)^2 \in M^0$. Since $n,n^2 \in M$, it follows that
$1 \in -M$, contradicting the assumption that $M$ is proper.
Therefore, $J_M \cap S = M^0$.

To prove that $J_M$ is prime, pick $a,b \in R$
such that $arb \in J_M$ for every $r \in R$.
If $b \not\in J_M$, then there exists $v \in R$
such that $b v v^\ast b^\ast \not\in M^0$.
Since $M$ is maximal and $-b v v^\ast b^\ast \not\in M$,
it follows that the smallest quadratic module containing
$M$ and $-b v v^\ast b^\ast$ is not proper. Hence,
there exist an element $m \in M$, an integer $k$
and elements $r_1,\ldots,r_k \in R$ such that
$1+m = \sum_{i=1}^k r_i (b v v^\ast b^\ast) r_i^\ast$.
Pick $r \in R$ and write $x = a r r^\ast a^\ast$ and $y = a r m r^\ast a^\ast$.
Since $a r r_i b \in J_M$ for every $i=1,\ldots,k$ by assumption and
$J_M \cap S = M^0$ by the first paragraph of this proof, it follows that
$x+y = a r (1+m) r^\ast a^\ast
= \sum_{i=1}^k (a r r_i b) v v^\ast (a r r_i b)^\ast \in M^0$.
Clearly, $x,y \in M$, so that $x,y \in M^0$.
Therefore, $a r r^\ast a^\ast$ for every $r \in R$,
implying that $a \in J_M$.

To show that $J_M$ is $\ast$-invariant, pick $a \in J_M$.
Since $J_M$ is an ideal, it follows that $a^\ast u u ^\ast a \in J_M$
for every $u \in R$. By the first paragraph of this proof, $J_M \cap S = M^0$,
so that $a^\ast u u ^\ast a \in M^0$ for every $u \in R$. So, $a^\ast \in J_M$
by the definition of $J_M$.
\end{proof}

\begin{rem}
If $R$ is a complex $\ast$-algebra and $M$ is a maximal proper
quadratic module in $R$, then $J_M = M^0+i M^0$. Namely, pick
$z \in J_M$ and write $z = x+iy$ where $x,y \in S$. Then $z^\ast
=x-iy$ also belongs to $J_M$. Pick any $r \in R$ and write $s = rr^\ast$.
Since $z s z^\ast \in M^0$ and $z^\ast s z \in M^0$, it follows that 
$2(xsx+ysy)=zsz^\ast+z^\ast sz \in M^0$. Since $xsx, ysy \in M$,
it follows that $xsx, ysy \in M^0$. Hence $x,y \in J_M$.
The relation $J_M \cap S = M^0$ implies that $x,y \in M^0$.
The converse is clear.
\end{rem}

\section{Quadratic modules and rings of fractions}
\label{fracsec}

We assume that the reader is familiar with the definition of a reversible Ore set
and Ore localization,  see Section 1.3 of \cite{cohn}.
The aim of this section is to discuss consequences of the following observation:

\begin{prop}
\label{ore}
Let $M$ be a proper quadratic module in a $\ast$-ring $A$ and let $N$ 
be a $\ast$-invariant reversible Ore set on $A$ such that $N \cap M^0 = \emptyset$.
Let $Q = AN^{-1}$ and $\tilde{M}=\{q \in \sym(Q) \vert \ n q n^\ast \in M$
for some $n \in N\}.$ Then $\tilde{M}$ is a proper quadratic module in $Q$. 
\end{prop}

\begin{proof} 
Clearly, the involution extends uniquely from $A$ to $Q$, see also \cite{domo}.
To prove that $\tilde{M}+\tilde{M} \subseteq \tilde{M}$
take $q_1,q_2 \in \tilde{M}$. Pick $n_1,n_2 \in N$ such that
$n_1 q_1 n_1^\ast \in M$ and $n_2 q_1 n_2^\ast \in M$. 
By the Ore property of $N$, there exist $u \in A$ and $v \in N$ 
such that $u n_1 = v n_2$. It follows that
$v n_2(q_1+q_2)(v n_2)^\ast = u(n_1 q_1 n_1^\ast)u^\ast
+v(n_2 q_1 n_2^\ast)v^\ast \in M$. Since $v n_2 \in N$,
we have that $q_1+q_2 \in \tilde{M}$.

Suppose that $q \in \tilde{M}$ and $d = n^{-1} a \in Q$.
Pick $z \in N$ such that $zqz^\ast \in M$. 
By the Ore property of $N$, there exist $b \in A$ and $w \in N$
such that $bz=wa$. Then, $(wn)(dqd^\ast)(wn)^\ast
=waq(wa)^\ast =b(zqz^\ast)b^\ast \in M$ and $wn \in N$.
Hence, $dqd^\ast \in \tilde{M}$.

If $-1 \in \tilde{M}$, then there exists $n \in N$ such that
$-nn^\ast \in M$. It follows that $nn^\ast \in M^0 \cap N$
contrary to the assumption that $M^0 \cap N = \emptyset$.
\end{proof}

Let $A$, $N$ and $Q$ be as in Proposition \ref{ore}. Then for every proper
quadratic module $M'$ in $Q$ we have 
\[
\big(M' \cap \sym(A) \big)\ \tilde{ } = M'. 
\]
On the other hand, for every quadratic module $M$ in $A$ such that 
$M^0 \cap N = \emptyset$, the set 
\[
\overline{M} := \tilde{M} \cap \sym(A) = \{a \in \sym(A) \vert \ n a n^\ast \in M
\text{ for some } n \in N\}
\]
is also a quadratic module in $A$ such that $\overline{M}^0 \cap N = \emptyset$,
which we call the $N$-\textit{closure} of $M$. We say that $M$ is $N$-\textit{closed} if
$M^0 \cap N = \emptyset$ and $M = \overline{M}$. 
Note that $N$-closure is an idempotent operation.

\begin{thm} 
\label{bij}
Let $A,N,Q$ be as above. The mappings $M \mapsto \tilde{M}$
and $M' \mapsto M' \cap \sym(A)$ give a bijective correspondence between 
proper $N$-closed quadratic modules of $A$ and proper quadratic modules in $Q$.
\end{thm}

\section{A representation theorem}
\label{repsec}

Suppose that $R$ is a prime Noetherian $\ast$-ring
and $N$ the set of all elements from $R$ that are not zero divisors.
A variant of Goldie's Theorem from \cite{domo} says that 
$N$ is a $\ast$-invariant reversible Ore set, the involution of $A$
extends uniquely to the Ore's localization $Q = RN^{-1}$ and 
there exists a skew--field $D$ such that $Q$ is either isomorphic to $M_n(D)$ 
or $\ast$-isomorphic to $M_n(D) \oplus M_n(D)^{op}$ 
with exchange involution $(a,b)^\ast=(b,a)$.
If $R$ is \textit{real}, i.e. it has a support zero quadratic module, then by Proposition \ref{ore}
this quadratic module extends to a proper quadratic module of $Q$.
Note that $M_n(D) \oplus M_n(D)^{op}$ with involution $(a,b)^\ast=(b,a)$ cannot 
have a proper quadratic module because of the identity $(-1,-1) = (1,-1)(1,-1)^\ast$. Hence:

\begin{prop}
The Goldie ring of fractions of a real prime Noetherian $\ast$-ring
is isomorphic to $M_n(D)$ for an integer $n$ and a skew-field $D$.
\end{prop}

The main result of this section is the following representation-theoretic 
characterization of maximal proper quadratic modules in Noetherian $\ast$-rings:

\begin{thm}
\label{rep}
Let $A$ be a Noetherian $\ast$-ring and $M$ a proper quadratic module in $A$.
The following are equivalent:
\begin{enumerate}
\item $M$ is maximal,
\item there exists a simple artinian ring with involution $Q \cong M_n(D)$,
a maximal proper quadratic module $M'$ in $Q$ and 
a $\ast$-ring homomorphism $\pi \colon A \to Q$ such that 
\[
M = \pi^{-1}(M') \cap \sym(A).
\]
\end{enumerate}
\end{thm}

\begin{proof} 
$(1) \Rightarrow (2) \colon$ 
If $M$ is maximal then by Theorem \ref{prime} $J_M$ is 
a prime $\ast$-ideal such that $J_M \cap \sym(A) = M^0$. 
Let $j \colon A \to A/J_M$ be the canonical projection. 
Then $j(M)$ is a maximal support zero quadratic module 
on $A/J_M$ and $M = j^{-1}(j(M)) \cap \sym(A)$. 
Let $N$ be the set of all non-zero divisors in $A/J_M$ and $Q = (A/J_M) N^{-1}$.
By Proposition \ref{ore} and Theorem \ref{bij}, $M' = j(M)\ \tilde{ }$
a maximal  proper quadratic module in $Q$, $j(M) = M' \cap j(A)$ 
and $Q$ is a simple artinian ring with involution.
Let $i \colon A/J_M \to Q$ be the canonical imbedding. 
Then $i^{-1}(M') = M' \cap j(A)= j(M)$.
Write $\pi = i \circ j$ and note that $\pi \colon A \to Q$ is a 
$\ast$-ring homomorphism such that $M = \pi^{-1}(M') \cap \sym(A)$.

$(2) \Rightarrow (1) \colon$ Follows from Theorem \ref{bij}.
\end{proof}

Theorem \ref{rep} reduces the study of maximal proper quadratic modules
in Noetherian $\ast$-rings into two parts: the study of their supports and 
the study of maximal proper quadratic modules in simple artinian rings with involution. 
In the special case of PI $\ast$-rings, the simple artinian rings in part two
are finite dimensional (i.e. central simple algebras). Quadratic modules 
on central simple algebras with involution will be studied in a separate paper.

\section{Intersection theorem}
\label{intsec}

Intersection theorems for quadratic modules in commutative rings with trivial involution
have been considered in \cite{jacobi}. Theorem \ref{nicht} generalizes \cite[1.8 Satz]{jacobi}.

\begin{thm}
\label{nicht}
Let $R$ be a semireal ring and $S$ its set of symmetric elements.
For every $x \in S$ and every proper quadratic module $M_0$ the
following are equivalent:
\begin{enumerate}
\item $x \in S \setminus -M$ for every maximal proper quadratic module $M$ containing $M_0$,
\item $-1 \in M_0 - N[x]$, where $M_0-N[x]$ is the smallest quadratic module containing $M_0$ and $-x$.
\end{enumerate}
\end{thm}

\begin{proof}
If (2) is false, then $-1 \not\in M_0-N[x]$, hence $M_0-N[x]$
is a proper quadratic module. By Zorn's Lemma, there exists a maximal
proper quadratic module $M$ containing $M_0-N[x]$. 
The fact that $x \in -M$ implies that (1) is false.
Conversely, if (1) is false, then $x \in -M$ for some maximal proper quadratic module $M$.
It follows that $M_0-N[x] \subseteq M$, so that $-1 \not\in M_0-N[x]$. Hence, (2) is false.
\end{proof}

Intersection theorems are popularly called \text\it{Stellens\" atze}. M. Schweighofer proposed the name
\textit{abstract Nirgendsnegativsemidefinitheitsstellensatz} for our Theorem \ref{nicht}.

In the archimedean case we can reformulate Theorem \ref{nicht}
by using the representation theorem from \cite{cstar}.
Recall that a proper quadratic module $M$
is \textit{archimedean} if for every element $a \in R$ there exists an
integer $n$ such that $n-aa^\ast \in M$. In Section 2 of \cite{cstar},
it is shown that for every archimedean quadratic module $M$ in a  
complex $\ast$-algebra $R$, there exists an $M$-positive irreducible 
$\ast$-representation $\phi_M$ of $R$ on a complex Hilbert space.
Recall that a representation $\phi$ is \textit{$M$-positive} if
$\phi(a)$ is positive semidefinite for every $a \in M$.

\begin{thm}
\label{another}
For every archimedean proper quadratic module $M_0$ in 
a complex $\ast$-algebra $R$ and for every $x \in \sym(R)$, 
the following are equivalent:
\begin{enumerate}
\item For every $M_0$-positive irreducible representation $\psi$ of $R$,
$\psi(x)$ is not negative semidefinite.
\item There exists $k \in \NN$ and $r_1,\ldots,r_k \in R$ 
such that $\sum_{i=1}^k r_i x r_i^\ast \in 1+M_0$.
\end{enumerate}
\end{thm}

\begin{proof}
If (1) is false, then there exists 
an $M_0$-positive irreducible $\ast$-representation $\phi$
such that $\phi(x)$ is negative semidefinite. 
Hence, $\phi$ is $M_0-N[x]$ positive. If follows that $-1 \not\in M_0-N[x]$,
so that (2) is false.

Conversely, if (2) is false then $-1 \not\in M_0-N[x]$,
so that $M_0-N[x]$ is a proper quadratic module containing $M_0$.
By Zorn's Lemma, there exists a maximal proper quadratic module
$M$ containing $M_0-N[x]$. Clearly, $M$ is archimedean.
By the discussion above, there exists an $M$-positive irreducible 
$\ast$-representation $\phi$ of $R$. Clearly, $\phi$ is $M_0$-positive
and $\phi(-x)$ is positive semidefinite. Hence, (1) is false.
\end{proof}

\section{An application of the intersection theorem}
\label{appsec}

The aim of this section is to prove a variant of Schm\" udgen's Strict 
Positivstellensatz for the Weyl algebra, see \cite[Theorem 1.1]{sch1}.
We will refer to Schm\" udgen's original proof several times.

Recall that the $d$-th Weyl algebra $\mathcal{W}(d)$ is the unital complex
\linebreak
$\ast$-algebra with generators $a_1,\ldots,a_d,a_{-1},\ldots,a_{-d}$, defining relations
\linebreak
$a_k a_{-k}-a_{-k} a_k=1$ for $k=1,\ldots,d$ and
$a_k a_l=a_l a_k$ for $k,l=1,\ldots,d$, $-1,\ldots,-d$, $k \ne -l$ and 
involution $a_k^\ast =a_{-k}$ for $k=1,\ldots,d$. Write $\deg$ for the 
total degree in the generators. The graded algebra that corresponds to the 
filtration by $\deg$  is $\CC[z,\bar{z}]=\CC[z_1,\ldots,z_d,\bar{z}_1,\ldots,\bar{z}_d]$, 
a polynomial algebra in $2d$ complex variables.
Write $N=a_1^\ast a_1+\ldots+a_d^\ast a_d$ and fix $\alpha \in \RR^+\setminus \NN$.
Let $\mathcal{N}$ be the set of all finite products of elements $N+(\alpha+n)1$,
where $n \in \ZZ$. Let $\Phi$ be the Fock-Bargmann representation of $\mathcal{W}(d)$.
It is unitarily equivalent to the Schr\" odinger representation $\pi_0$.
 Let $\mathcal{X}$ be the unital complex $\ast$-algebra 
generated by $y_n=\Phi(N+(\alpha+n)1)^{-1}$ for $n \in \ZZ$ and $x_{kl}=\Phi(a_k a_l) y_0$,
$k,l=1,\ldots,d,-1,\ldots,-d$. By \cite[Lemma 3.1]{sch1}, $N(\mathcal{X})$ is an
archimedean quadratic module in $\mathcal{X}$.

\begin{thm} 
Let $c$ be  a symmetric element of $\mathcal{W}(d)$ with 
even degree $2m$ and let $c_{2m}$ be the polynomial from
$\CC[z,\bar{z}]$ that 
corresponds to the $2m$-th component of $c$. The following
assertions are equivalent:
\begin{enumerate}
\item $\pi_0(c)$ is not negative semidefinite and 
$c_{2m}(z,\overline{z}) > 0$ for all $z \in \mathbb{C}^d$, $z \ne 0$.
\item There exist $k,l \in \NN$ and elements $r_1,\ldots,r_k,s_0,\ldots,s_l \in \mathcal{W}(d)$
such that $\sum_{i=1}^k r_i^\ast c r_i = \sum_{j=0}^l s_j^\ast s_j$ and $s_0 \in \mathcal{N}$
and $\deg(s_0) \ge \deg(s_j)$ for every $j=1,\ldots,l$.
\end{enumerate}
\end{thm}

\begin{proof}
$(1) \Rightarrow (2) \colon$
If $c$ has degree $4n$ then by \cite[Lemma 3.2]{sch1} the element
$\tilde{c} := y_0^n \Phi(c) y_0^n$ belongs to $\mathcal{X}$.
The main part of the proof is to show that there exist elements 
$h_i \in \mathcal{X}$ such that $\sum_i h_i \tilde{c} h_i^\ast \in 1+ N( \mathcal{X} )$. 
(Then we get (2) by clearing out the denominators using the identities 
$\Phi(a_j)y_k=y_{k+1}\Phi(a_j)$ and $\Phi(a_j)^\ast y_k=y_{k-1} \Phi(a_j)^\ast$.) 
If this claim is false, then by Theorem \ref{another},
there exists a representation $\pi$ of $\mathcal{X}$ such that 
$\pi(\tilde{c})$ is negative semidefinite. By \cite[Section 4]{sch1}, we know that 
$\pi$ can be decomposed as $\pi_1 \oplus \pi_\infty$ where $\pi_1$ is a sum 
of \lq\lq identity\rq\rq{} representations and $\pi_\infty$ is an integral representation 
defined by $\pi_\infty(\tilde{c}) = \int_{S^d} c(z,\bar z) dE(z,\bar z)$
where $E$ is a spectral measure on the sphere $S^d$.
By \cite[Section 5]{sch1}, we have that $\pi_\infty(\tilde{c}) = \pi_\infty(c_{4n})$
where $c_{4n}$ is the leading term of $c$. Since $c_{4n}(z,\bar z) > 0$
for every $z \in S^d$ by the second assumption, 
there exists by the compactness of $S^d$ 
an $\epsilon > 0$ such that $c_{4n} \ge \epsilon$. 
It follows that $\langle \pi_\infty(\tilde{c}) \phi, \phi \rangle  
\ge \int_{S^d} \epsilon d \langle E(z,\bar z) \phi, \phi \rangle
= \epsilon \int_{S^d} d \Vert E(z,\bar z) \phi \Vert^2 = 
\epsilon \Vert \phi \Vert^2$ for every $\phi \in L^2(\RR^d)$.
Since $\pi(\tilde c)$ is negative semidefinite and $\pi_\infty(\tilde c)$ 
is positive definite, it follows that $\pi_1(\tilde c)$ is nontrivial and 
negative definite. Since $\pi_1$ is a direct sum of identity representations,
it follows that $\tilde c < 0$. Since $y_0^n$ has dense image, 
it follows that $\Phi(c) \le 0$. Hence $\pi_0(c) \le 0$ by the
unitary equivalence of $\Phi$ and $\pi_0$, a contradiction
with assumption (1).
If $c$ has degree $4n+2$, then we replace $c$ by $\sum_{j=1}^d a_j c a_j^\ast$
and proceed as above.

$(2) \Rightarrow (1) \colon$ 
Since $s_0 \in \mathcal{N}$, $\pi_0(s_0)$ is invertible.
It follows that $\pi_0(\sum_{j=0}^l s_j^\ast s_j)>0$. 
Since $\sum_{i=1}^k r_i^\ast c r_i= \sum_{j=0}^l s_j^\ast s_j$ it follows that also
$\pi_0(\sum_{i=1}^k r_i^\ast c r_i)>0$. Hence $c \not\le 0$ as claimed.
Write $t=\deg(s_0)$ and note that 
$(s_0^\ast s_0)_{2t}(z,\bar{z}) = (\sum_{n=1}^d z_n \bar{z}_n)^t > 0$ for $z \ne 0$,
Since $t \ge \deg(s_j)$ for every $j=1,\ldots,l$,
it follows that $(\sum_{j=0}^l s_j^\ast s_j)_{2t}(z,\bar{z}) >0$ for $z \ne 0$.
From $\sum_{i=1}^k r_i^\ast c r_i= \sum_{j=0}^l s_j^\ast s_j$ and 
$
(\sum_{i=1}^k r_i^\ast c r_i)_{2t}(z,\overline{z})=
(\sum_{i=1}^k r_i^\ast r_i)_{2t-2m}(z,\overline{z})
c_{2m}(z,\overline{z})
$
we get that $c_{2m}(z,\overline{z}) > 0$ for $z\ne 0$.
\end{proof}

\begin{rem}
In \cite[Theorem 1.1]{sch1} both (1) and (2) are stronger.
In (1) the assumption $\pi_0(c) \not\le 0$ is replaced by $\pi_0(c-\epsilon \cdot 1)>0$
for some $\epsilon$ and in (2) the bound $k \le d$ is provided. The implication
from (2) to (1) is not considered.
\end{rem}

We believe that the main result of \cite{sch2} can be extended in a similar way.
It would also be interesting to know whether the second part of assertion (1)
($c_{2m}>0$) implies the first part ($\pi_0(c) \not\le 0$).


\begin{thebibliography}{99}

\bibitem{av} C.-G. Ambrozie, F.-H. Vasilescu,
\textit{Operator-theoretic positivstellens\" atze},
Z. Anal. Anwendungen  \textbf{22}  (2003),  no. 2, 299--314.
\bibitem{cstar} J. Cimpri\v c, 
\textit{A representation theorem for archimedean quadratic modules on $\ast$-rings}, 
to appear in Canad. Math. Bull.
\bibitem{cohn} P. M. Cohn,
Skew Fields, Theory of General Division Rings, 
Cambridge University Press 1995, ISBN 0-521-43217-0.
\bibitem{mcr} J. C. McConnell, J. C. Robson,  
Noncommutative Noetherian Rings, revised edition, 
Graduate Studies in Mathematics 30, 
American Mathematical Society, Providence, RI, 2001, xx+636 pp, ISBN: 0-8218-2169-5.
\bibitem{pc} S. McCullough, M. Putinar,
\textit{Noncommutative sums of squares},
Pacific J. Math. \textbf{218} (2005), no. 1, 167--171.
\bibitem{domo} M. Domokos,
\textit{Goldie's theorems for involution rings},
Comm. Algebra \textbf{22} (1994), no. 2, 371--380.
\bibitem{ha} D. Handelman,
\textit{Rings with involution as partially ordered abelian groups},
Rocky Mountain J. Math.  \textbf{11}  (1981), no. 3, 337--381. 
\bibitem{h1} J. W. Helton, 
\textit{"Positive" noncommutative polynomials are sums of squares},
Ann. of Math. \textbf{156}  (2002),  no. 2, 675--694.
\bibitem{jacobi} T. Jacobi, 
\" Uber die Darstellung positiver Polynome auf
semi-algebraischen Kompakta, Ph. D. Thesis, University of Konstanz, 1999.
\bibitem{ks} I. Klep, M. Schweighofer,
\textit{A Nichtnegativstellensatz for polynomials in noncommuting variables}, 
preprint at \texttt{ihp-raag.org}.
\bibitem{lam1} T. Y. Lam,
A First Course in Noncommutative Rings, second edition, 
Graduate Texts in Mathematics 131,
Springer-Verlag, New York, 2001, xx+385 pp, ISBN: 0-387-95183-0.
\bibitem{lam2} T. Y. Lam,
Lectures on Modules and Rings,
Graduate Texts in Mathematics 189,
Springer-Verlag, New York, 1999, xxiv+557 pp, ISBN: 0-387-98428-3.
\bibitem{mm} M. Marshall,
\textit{$*$-orderings on a ring with involution},
Comm. Algebra  \textbf{28}  (2000),  no. 3, 1157--1173.
\bibitem{pv} M. Putinar, F.-H. Vasilescu,
\textit{Solving moment problems by dimensional extension},
Ann. of Math. \textbf{149}  (1999),  no. 3, 1087--1107.
\bibitem{schbook} K. Schm\" udgen,
Unbounded Operator Algebras and Representation Theory,
Operator Theory: Advances and Applications 37,
Birkhäuser Verlag, Basel, 1990, 380 pp, ISBN: 3-7643-2321-3
\bibitem{sch1} K. Schm\" udgen, 
\textit{A strict Positivstellensatz for the Weyl algebra}, 
Math. Ann. \textbf{331} (2005), no. 4, 779--794.
\bibitem{sch2} K. Schm\" udgen, 
\textit{A strict Positivstellensatz for enveloping algebras}, 
preprint at \texttt{arxiv.org}.
 \bibitem{ns} N. Schwartz,
\textit{About Schm\" udgen's theorem},
preprint at \texttt{ihp-raag.org}.
\bibitem{walter} L. J. Walter,
Orders and Signatures of Higher Level on a Commutative Ring,
Ph. D. Thesis, University of Saskatchewan, 1994.
\end{thebibliography}
\end{document}